\documentclass[a4paper,twoside,12pt]{article}
\usepackage[english]{babel}

\usepackage{amsmath,amsfonts,amssymb,amsthm}
\newtheorem{teo}{Theorem}
\newtheorem{lem}{Lemma}

\newtheorem{corol}{Corollary}

 \title{{\bf One more  theorem on norm equivalence  in the Lebesgue  space      }}

\author{Maksim \,V.~Kukushkin   \\ \\
 \small  \textit{Moscow State University of Civil Engineering, 129337,  Moscow, Russia}\\
 \small\textit{Kabardino-Balkarian Scientific Center, RAS, 360051,  Nalchik, Russia}\\
\textit{\small\textit{kukushkinmv@rambler.ru}} }

\date{}
\begin{document}

\maketitle

\begin{abstract}
In this paper we consider a   norm based on the infinitesimal generator of the shift semigroup in a direction. The relevance of such a focus is guaranteed by an  abstract representation of a fractional integro-differential operator  by means of a composition of the  corresponding infinitesimal generator.
The main result of the paper is a theorem establishing  equivalence of norms in  functional spaces. Even without mentioning the relevance  of this result for the constructed theory, we claim  it deserves to be considered   itself.

\end{abstract}
\begin{small}\textbf{Keywords:}  Equivalence of norms; compact embedding of spaces;  infinitesimal generator;
 m-accretive operator; uniformly elliptic operator.\\\\
{\textbf{MSC} 46E30; 46E40;     47B28; 20M05. }
\end{small}

\section{Introduction}

The facts that have motivated us to write this paper lie in the  fractional calculus theory. Basically, an event that a differential operator with a fractional derivative in final terms underwent  a careful study \cite{firstab_lit:1Nakhushev1977}, \cite{kukushkin2019}  have played an important role in our research. The main feature is that there exists various approaches to study the operator and one of them is based on an opportunity to represent it in a sum of a senior term and an a lower term, here we should note that this method works if the  senior term is selfadjoint or normal. Thus,
in the case corresponding to a selfadjoint senior term, we can partially solve the problem  having applied the results of the   perturbation theory,   within the framework of which  the following papers are well-known   \cite{firstab_lit:1Katsnelson}, \cite{firstab_lit:1Krein},   \cite{firstab_lit:2Markus},
  \cite{firstab_lit:3Matsaev}, \cite{firstab_lit:1Mukminov},
 \cite{firstab_lit:Shkalikov A.}. Note that  to apply the last paper results  we must have the mentioned above representation. In other cases we can use methods of the papers  \cite{firstab_lit(arXiv non-self)kukushkin2018},  which are relevant if we deal with non-selfadjoint operators and allow us  to study spectral properties of   operators.
In the  paper \cite{kukushkin Gen}  we  explore  a special  operator class for which    a number of  spectral theory theorems can be applied. Further, we construct an abstract  model of a  differential operator in terms of m-accretive operators and call it an m-accretive operator transform, we  find  such conditions that    being  imposed guaranty  that the transform   belongs to the class. One of them is a compact embedding of a space generated by an m-accretive operator (infinitesimal generator) into the initial Hilbert space. Note that in the case corresponding to the second  order operator with the  Kiprianov operator in final terms we have obtained  the embedding mentioned above in the one-dimensional case only. In this paper we try to reveal this problem  and the main result   is a theorem establishing    equivalence of norms in   function spaces  in consequence  of  which we have  a   compact embedding of a space generated by the infinitesimal generator of the shift semigroup in a direction into the Lebsgue space. We should note that this result do not give us a useful concrete  application in the built theory for it is more of  an abstract generalization.  However   this result, by virtue of popularity and well known applicability of the  Lebesgue spaces theory,    deserves    to be considered    itself.

\section{Preliminaries}

Let    $ C,C_{i} ,\;i\in \mathbb{N}_{0}$ be   real constants. We   assume   that  a  value of $C$ is positive and   can be different in   various formulas  but   values of $C_{i} $ are  certain.   Everywhere further, if the contrary is not stated, we consider   linear    densely defined operators acting on a separable complex  Hilbert space $\mathfrak{H}$. Denote by $ \mathcal{B} (\mathfrak{H})$    the set of linear bounded operators   on    $\mathfrak{H}.$  Denote by
    $\tilde{L}$   the  {\it closure} of an  operator $L.$  Denote by    $    \mathrm{D}   (L),\,   \mathrm{R}   (L),\,\mathrm{N}(L)$      the  {\it domain of definition}, the {\it range},  and the {\it kernel} or {\it null space}  of an  operator $L$ respectively.
Consider a pair of complex Hilbert spaces $\mathfrak{H},\mathfrak{H}_{+},$ the notation
$
\mathfrak{H}_{+}\subset\subset\mathfrak{ H}
$
   means that $\mathfrak{H}_{+}$ is dense in $\mathfrak{H}$ as a set of    elements and we have a bounded embedding provided by the inequality
$
\|f\|_{\mathfrak{H}}\leq C_{0}\|f\|_{\mathfrak{H}_{+}},\,C_{0}>0,\;f\in \mathfrak{H}_{+},
$
moreover   any  bounded  set with respect to the norm $\mathfrak{H}_{+}$ is compact with respect to the norm $\mathfrak{H}.$
 An operator $L$ is called  {\it bounded from below}   if the following relation  holds  $\mathrm{Re}(Lf,f)_{\mathfrak{H}}\geq \gamma_{L}\|f\|^{2}_{\mathfrak{H}},\,f\in  \mathrm{D} (L),\,\gamma_{L}\in \mathbb{R},$  where $\gamma_{L}$ is called a lower bound of $L.$ An operator $L$ is called  {\it   accretive}   if  $\gamma_{L}=0.$
 An operator $L$ is called  {\it strictly  accretive}   if  $\gamma_{L}>0.$      An  operator $L$ is called    {\it m-accretive}     if the next relation  holds $(A+\zeta)^{-1}\in \mathcal{B}(\mathfrak{H}),\,\|(A+\zeta)^{-1}\| \leq   (\mathrm{Re}\zeta)^{-1},\,\mathrm{Re}\zeta>0. $
 Assume that  $T_{t},\,(0\leq t<\infty)$ is a semigroup of bounded linear operators on    $\mathfrak{H},$    by definition put
$$
Af=-\lim\limits_{t\rightarrow+0} \left(\frac{T_{t}-I}{t}\right)f,
$$
where $\mathrm{D}(A)$  is a set of elements for which  the last limit  exists in the sense of the norm  $\mathfrak{H}.$    In accordance with definition \cite[p.1]{Pasy},\cite{firstab_lit:Yosida} the operator $-A$ is called the  {\it infinitesimal  generator} of the semigroup $T_{t}.$
Using     notations of the paper     \cite{firstab_lit:kipriyanov1960} we assume that $\Omega$ is a  convex domain of the  $n$ -  dimensional Euclidean space $\mathbb{E}^{n}$, $P$ is a fixed point of the boundary $\partial\Omega,$
$Q(r,\mathbf{e})$ is an arbitrary point of $\Omega;$ we denote by $\mathbf{e}$   a unit vector having a direction from  $P$ to $Q,$ denote by $r=|P-Q|$   the Euclidean distance between the points $P,Q,$ and   use the shorthand notation    $T:=P+\mathbf{e}t,\,t\in \mathbb{R}.$
We   consider the Lebesgue  classes   $L_{p}(\Omega),\;1\leq p<\infty $ of  complex valued functions.  For the function $f\in L_{p}(\Omega),$    we have
\begin{equation}\label{1}
\int\limits_{\Omega}|f(Q)|^{p}dQ=\int\limits_{\omega}d\chi\int\limits_{0}^{d(\mathbf{e})}|f(Q)|^{p}r^{n-1}dr<\infty,
\end{equation}
where $d\chi$   is an element of   solid angle of
the unit sphere  surface (the unit sphere belongs to $\mathbb{E}^{n}$)  and $\omega$ is a  surface of this sphere,   $d:=d(\mathbf{e})$  is the  length of the  segment of the  ray going from the point $P$ in the direction
$\mathbf{e}$ within the domain $\Omega.$
 We use a shorthand  notation  $P\cdot Q=P^{i}Q_{i}=\sum^{n}_{i=1}P_{i}Q_{i}$ for the inner product of the points $P=(P_{1},P_{2},...,P_{n}),\,Q=(Q_{1},Q_{2},...,Q_{n})$ which     belong to  $\mathbb{E}^{n}.$
     Denote by  $D_{i}f$  a weak partial derivative of the function $f$ with respect to a coordinate variable with index   $1\leq i\leq n.$
We  assume that all functions have  a zero extension outside  of $\bar{\Omega}.$
Everywhere further,   unless  otherwise  stated,  we   use  notations of the papers      \cite{firstab_lit:kato1980},  \cite{firstab_lit:kipriyanov1960}, \cite{firstab_lit:1kipriyanov1960}.

\begin{lem}\label{L1}
Assume that   $A$ is  a  closed densely defined  operator, the following    condition  holds
\begin{equation}\label{2}
\|(A+t)^{-1}\|_{\mathrm{R} \rightarrow \mathfrak{H}}\leq\frac{1}{ t},\,t>0,
\end{equation}
where a notation  $\mathrm{R}:=\mathrm{R}(A+t)$ is used. Then the operator  $A$ is m-accretive.
\end{lem}
\begin{proof}
Using \eqref{2} consider
$$
\|f\|^{2}_{\mathfrak{H}}\leq  \frac{1}{t^{2}}  \|(A+t)f \|^{2}_{ \mathfrak{H}};\,
 \|f\|^{2}_{\mathfrak{H}}\leq  \frac{1}{ t^{2}} \left\{ \| A f\|^{2}_{ \mathfrak{H}}+2t \mathrm{Re}(Af,f)_{ \mathfrak{H}}+t^{2}\|   f \|^{2}_{ \mathfrak{H}}\right\} ;
$$
 $$
t^{-1} \| A  f \|^{2}_{ \mathfrak{H}}+2 \mathrm{Re}(A f,f)_{ \mathfrak{H}}\geq0,\,f\in  \mathrm{D} (A).
$$
Let $t$ be tended to infinity, then we obtain
\begin{equation}\label{3}\mathrm{Re}(A f,f)_{ \mathfrak{H}}\geq0 ,\,f\in \mathrm{D}(A).
\end{equation}
It means  that   the  operator $A$ is   accretive.
Due to \eqref{3}, we have   $\{\lambda \in \mathbb{C}:\,\mathrm{Re}\lambda<0\}\subset \Delta(A),  $ where $\Delta(A)=\mathbb{C}\setminus   \overline{\Theta  (A)}.$  Applying Theorem 3.2 \cite[p.268]{firstab_lit:kato1980}, we obtain that $A-\lambda$ has a closed range and  $\mathrm{nul} (A-\lambda)=0,\,\mathrm{def} (A-\lambda)=\mathrm{const},\,\forall\lambda\in \Delta(A).$
Let $\lambda_{0}\in \Delta(A) ,\;{\rm Re}\lambda_{0} <0.$
Note that in consequence of inequality  \eqref{3}, we have
 \begin{equation}\label{4}
  {\rm Re} ( f,(A-\lambda )f  )_{\mathfrak{H}}\geq   - {\rm Re} \lambda   \|f\|^{2}_{\mathfrak{H}},\,f\in \mathrm{D}(A).
 \end{equation}
  Since the  operator $A-\lambda_{0}$ has a closed range, then
\begin{equation*}
 \mathfrak{H}=\mathrm{R} (A-\lambda_{0})\oplus \mathrm{R} (A-\lambda_{0})^{\perp} .
 \end{equation*}
We remark  that the intersection of the sets  $\mathrm{D}(A)$ and $\mathrm{R} (A-\lambda_{0})^{\perp}$ is  zero, because if we assume  the contrary,   then applying inequality  \eqref{4},  for arbitrary element
 $f\in \mathrm{D}(A)\cap \mathrm{R}  (A-\lambda_{0})^{\perp}$    we get
 \begin{equation*}
 - {\rm Re} \lambda_{0}  \|f\|^{2}_{\mathfrak{H}} \leq  {\rm Re} ( f,[A-\lambda_{0} ]f  )_{\mathfrak{H}}=0,
 \end{equation*}
hence $f=0.$   It implies that
$$
\left(f,g\right)_{\mathfrak{H}}=0,\;\forall f\in  \mathrm{R}  (A-\lambda_{0})^{\perp},\;\forall g\in \mathrm{D}(A).
$$
Since $  \mathrm{D}(A)$  is a dense set in $\mathfrak{H},$ then $\mathrm{R}  (A-\lambda_{0})^{\perp}=0.$ It implies that  ${\rm def} (A-\lambda_{0}) =0$ and if we take into account  Theorem 3.2 \cite[p.268]{firstab_lit:kato1980}, then
we come to the conclusion that ${\rm def} (A-\lambda )=0,\;\forall\lambda\in \Delta(A),$ hence the operator $A$ is m-accretive.
The proof is complete.
\end{proof}
Assume that  $\Omega\subset \mathbb{E}^{n}$ is  a convex domain, with a sufficient smooth boundary ($C^{3}$ class)  of n-dimensional Euclidian space. For the sake of the simplicity we consider that $\Omega$ is bounded.
Consider the shift semigroup in a direction acting on $L_{2}(\Omega)$ and  defined as follows
$
T_{t}f(Q)=f(Q+\mathbf{e}t),
$
where $Q\in \Omega,\,Q=P+\mathbf{e}r.$   The following lemma establishes  a property of  the infinitesimal generator $-A$ of the semigroup $T_{t}.$
\begin{lem}\label{L2} We claim that
$
A=\tilde{A_{0}},\,\mathrm{N}(A)=0,
$
where $A_{0}$  is a restriction of $A$ on the set
   $ C^{\infty}_{0}( \Omega ).$
 \end{lem}
\begin{proof}
Let us show  that $T_{t}$ is a strongly continuous semigroup ($C_{0}$  semigroup). It can be easily established   due to the  continuous in average property. Using the Minkowskii inequality, we have
 $$
 \left\{\int\limits_{\Omega}|f(Q+\mathbf{e}t)-f(Q)|^{2}dQ\right\}^{\frac{1}{2}}\leq  \left\{\int\limits_{\Omega}|f(Q+\mathbf{e}t)-f_{m}(Q+\mathbf{e}t)|^{2}dQ\right\}^{\frac{1}{2}}+
 $$
 $$
 +\left\{\int\limits_{\Omega}|f(Q)-f_{m}(Q)|^{2}dQ\right\}^{\frac{1}{2}}+\left\{\int\limits_{\Omega}|f_{m}(Q)-f_{m}(Q+\mathbf{e}t)|^{2}dQ\right\}^{\frac{1}{2}}=
 $$
 $$
 =I_{1}+I_{2}+I_{3}<\varepsilon,
 $$
where $f\in L_{2}(\Omega),\,\left\{f_{n}\right\}_{1}^{\infty}\subset C_{0}^{\infty}(\Omega);$  $m$ is chosen so that $I_{1},I_{2}< \varepsilon/3 $ and $t$
is chosen so that $I_{3}< \varepsilon/3.$
Thus,  there exists such a positive  number $t_{0}$   so that
$$
\|T_{t}f-f\|_{L_{2} }<\varepsilon,\,t<t_{0},
$$
for arbitrary small $\varepsilon>0.$ Hence in accordance with the definition $T_{t}$ is  a $C_{0}$ semigroup.   Using the assumption that  all functions have the zero extension outside $\bar{\Omega},$   we have
$\|T_{t}\|  \leq 1.$ Hence  we conclude that $T_{t}$ is a $C_{0}$ semigroup of contractions (see \cite{Pasy}).
Hence   by virtue of   Corollary 3.6 \cite[p.11]{Pasy}, we have
\begin{equation} \label{5}
\|(\lambda+A)^{-1}\|  \leq \frac{1}{\mathrm{Re} \lambda },\,\mathrm{Re}\lambda>0.
\end{equation}
Inequality \eqref{5} implies that $A$ is m-accretive.
    It is the well-known fact that   an infinitesimal  generator    $-A$ is a closed operator, hence $A_{0}$  is closeable.
 It is not hard to prove that $ \tilde{A} _{0}$ is an m-accretive operator. For this purpose let us    rewrite relation \eqref{5} in the form
\begin{equation*}
\|(\lambda+ \tilde{A} _{0})^{-1}\|_{\mathrm{R}\rightarrow \mathfrak{H}}  \leq \frac{1}{\mathrm{Re} \lambda },\,\mathrm{Re}\lambda>0,
\end{equation*}
applying Lemma \ref{1}, we obtain that $ \tilde{A} _{0}$ is an m-accretive operator. Note that there does not exist an  accretive extension of an m-accretive operator (see \cite{firstab_lit:kato1980}). On the other hand it is clear that $\tilde{A} _{0}\subset A.$  Thus we conclude that $\tilde{A} _{0}= A.$
Consider an operator
 $$
 B f(Q)=\!\int_{0}^{r}\!\!f(P+\mathbf{e }[r-t])dt,\,f\in L_{2}(\Omega).
 $$
 It is not hard to prove that  $B \in \mathcal{B}(L_{2}),$   applying the generalized Minkowskii inequality, we get
 $$
 \|B f\|_{L_{2} }\leq \int\limits_{0}^{\mathrm{diam\,\Omega}}dt  \left(\int\limits_{\Omega}|f(P+\mathbf{e }[r-t])|dQ\right)^{1/2}\leq C\|f\|_{L_{2} }.
 $$
Note that   the fact $A^{-1}_{ 0}\subset B ,$    follows from the properties of the  one-dimensional integral defined on smooth functions.
Using  Theorem 2 \cite[p.555]{firstab_lit:Smirnov5}, the proved above fact $\tilde{A} _{0}= A,$  we deduce that  $A^{-1} =   \widetilde{A ^{-1}_{ 0}}  ,$    hence $A^{-1} \subset B .$
 The proof is complete.
\end{proof}

\section{Main results}
Consider a linear  space
$
\mathbb{L}^{n}_{2}(\Omega):=\left\{f=(f_{1},f_{2},...,f_{n}),\,f_{i}\in L_{2}(\Omega)\right\},
$
endowed with the inner product
$$
(f,g)_{\mathbb{L}^{n}_{2}}=\int\limits_{\Omega} (f, g)_{\mathbb{E}^{n}} dQ,\,f,g\in \mathbb{L}^{n}_{2}(\Omega).
$$
It is clear that this  pair forms a Hilbert space and let us use the same  notation $\mathbb{L}^{n}_{2}(\Omega)$ for it.
Consider a semilinear form
$$
t(f,g):=\sum\limits_{i=1}^{n}\int\limits_{\Omega} (f ,\mathbf{e_{i}})_{\mathbb{E}^{n}}\overline{(g,\mathbf{e_{i}})}_{\mathbb{E}^{n}} dQ,\,f,g\in \mathbb{L} ^{n}_{2} (\Omega),
$$
where $\mathbf{e_{i}}$ corresponds to $P_{i}\in \partial\Omega,\,i=1,2,...,n.$
\begin{lem}\label{L3}
The points $P_{i}\in \partial\Omega,\,i=1,2,...,n$ can be chosen so that the   form $t$ generates an inner product.
\end{lem}
\begin{proof}
It is clear that we should only establish an implication  $t(f,f)=0\,\Rightarrow f=0.$
 Since $\Omega\in \mathbb{E}^{n},$  then without lose of generality we can assume that   there exists $P_{i}\in \partial\Omega,\,i=1,2,...,n,$ such that
\begin{equation}\label{6}\Delta= \left|
\begin{array}{cccc}
P_{11}&P_{12}&...&P_{1n}\\
P_{21}&P_{22}&...&P_{2n}\\
...&...&...&...\\
P_{n1}&P_{n2}&...&P_{nn}
\end{array}
\right|\neq0,
\end{equation}
where $P_{i}=(P_{i1},P_{i2},...,P_{in}).$ It becomes clear if we remind that in the contrary case, for arbitrary set of points $P_{i}\in \partial\Omega,\,i=1,2,...,n,$  we   have
$$
P_{n}=\sum\limits_{k=1}^{n-1}c_{k}P_{k},\,c_{k}= \mathrm{const},
$$
from what follows  that we can consider $\Omega$ at least  as a subset of  $\mathbb{E}^{n-1}.$ Continuing this line of reasonings we can find such  a dimension $p$ that a corresponding $\Delta\neq0$ and further assume  that  $\Omega\in \mathbb{E}^{p}. $
 Consider a relation
\begin{equation*}
 \sum\limits_{i=1}^{n}\int\limits_{\Omega}| (\psi,\mathbf{e_{i}})_{\mathbb{E}^{n}}|^{2}   dQ  =0,\,\psi\in \mathbb{L}^{n}_{2}(\Omega).
\end{equation*}
It follows that  $\left(\psi(Q),\mathbf{e_{i}}\right)_{\mathbb{E}^{n}}=0$ a.e. $i=1,2,...,n.$  Note that every $P_{i}$ corresponds to the set
 $\vartheta_{i}:=\{Q\subset \vartheta_{i} :\;(\psi(Q),\mathbf{e_{i}})_{\mathbb{E}^{n}}\neq0  \}.$ Consider $\Omega'=\Omega\backslash\bigcup\limits_{i=1}^{n}\vartheta_{i},$ it is clear that $\mathrm{mess} \left(\bigcup\limits_{i=1}^{n}\vartheta_{i}\right)=0.$ Note that due to the made construction, we can reformulate the obtained above  relation in the coordinate form
\begin{equation*}
 \left\{ \begin{aligned}
(P_{11}-Q_{1})\psi_{1}(Q)+(P_{12}-Q_{2})\psi_{2}(Q)+...+(P_{1n}-Q_{n})\psi_{n}(Q)=0 \\
 (P_{21}-Q_{1})\psi_{1}(Q)+(P_{22}-Q_{2})\psi_{2}(Q)+ ... +(P_{2n}-Q_{n})\psi_{n}(Q)=0 \\
...\;\;\;\;\;\;\;\;\;\;\;\;\;\;\;\;\;\;\;\;\;\;\;\;\;\;\;\;...\;\;\;\;\;\;\;\;\;\;\;\;\;\;\;\;\;\;\;\;\;\;\;\;\;\;\;\; ...\;\;\;\;\;\;\;\;\;\;\;\;\;\;\;\;\;\;\;\;\;\;\;\;\;\;\;\; \\
 (P_{n1}-Q_{1})\psi_{1}(Q)+(P_{n2}-Q_{2})\psi_{2}(Q)+ ... +(P_{nn}-Q_{n})\psi_{n}(Q)=0
\end{aligned}
 \right.\;,
\end{equation*}
where $\psi=(\psi_{1},\psi_{2},...,\psi_{n}),\,Q=(Q_{1},Q_{2},...,Q_{n}),\, Q\in \Omega'.$ Therefore, if we prove that
$$\Lambda(Q)= \left|
\begin{array}{cccc}
P_{11}-Q_{1}&P_{12}-Q_{2}&...&P_{1n}-Q_{n}\\
P_{21}-Q_{1}&P_{22}-Q_{2}&...&P_{2n}-Q_{n}\\
...&...&...&...\\
P_{n1}-Q_{1}&P_{n2}-Q_{2}&...&P_{nn}-Q_{n}
\end{array}
\right|\neq0\;a.e.,
$$
then we obtain $\psi =0$ a.e. Assume the contrary i.e. that there exists such a set $ \Upsilon \subset  \Omega ,\,\mathrm{mess}\,\Upsilon \neq0,$ so that   $\Lambda(Q)=0,\,Q\in \Upsilon  .$   We have
$$ \left|
\begin{array}{cccc}
P_{11}-Q_{1}&P_{12}-Q_{2}&...&P_{1n}-Q_{n}\\
P_{21}-Q_{1}&P_{22}-Q_{2}&...&P_{2n}-Q_{n}\\
...&...&...&...\\
P_{n1}-Q_{1}&P_{n2}-Q_{2}&...&P_{nn}-Q_{n}
\end{array}
\right|=
\left|
\begin{array}{cccc}
P_{11} &P_{12} &...&P_{1n} \\
P_{21}-Q_{1}&P_{22}-Q_{2}&...&P_{2n}-Q_{n}\\
...&...&...&...\\
P_{n1}-Q_{1}&P_{n2}-Q_{2}&...&P_{nn}-Q_{n}
\end{array}
\right|-
$$
$$
-\left|
\begin{array}{cccc}
 Q_{1}& Q_{2}&...& Q_{n}\\
P_{21}-Q_{1}&P_{22}-Q_{2}&...&P_{2n}-Q_{n}\\
...&...&...&...\\
P_{n1}-Q_{1}&P_{n2}-Q_{2}&...&P_{nn}-Q_{n}
\end{array}
\right|=
\left|
\begin{array}{cccc}
P_{11} &P_{12} &...&P_{1n} \\
P_{21} &P_{22} &...&P_{2n} \\
...&...&...&...\\
P_{n1}-Q_{1}&P_{n2}-Q_{2}&...&P_{nn}-Q_{n}
\end{array}
\right|-
$$
$$
-\left|
\begin{array}{cccc}
P_{11} &P_{12} &...&P_{1n} \\
 Q_{1}& Q_{2}&...& Q_{n}\\
...&...&...&...\\
P_{n1}-Q_{1}&P_{n2}-Q_{2}&...&P_{nn}-Q_{n}
\end{array}
\right|
 -\left|
\begin{array}{cccc}
 Q_{1}& Q_{2}&...& Q_{n}\\
P_{21} &P_{22} &...&P_{2n} \\
...&...&...&...\\
P_{n1}-Q_{1}&P_{n2}-Q_{2}&...&P_{nn}-Q_{n}
\end{array}
\right|=
$$
$$
 =\left|
\begin{array}{cccc}
 P_{11} &P_{12} &...&P_{1n}\\
P_{21} &P_{22} &...&P_{2n} \\
...&...&...&...\\
P_{n1} &P_{n2} &...&P_{nn}
\end{array}
\right|-\sum\limits_{j=1}^{n} \Delta_{j}=0,
$$
where
$$
 \Delta_{j}=\left|
\begin{array}{cccc}
 P_{11} &P_{12} &...&P_{1n}\\
P_{21} &P_{22} &...&P_{2n} \\
...&...&...&...\\
P_{j-1\,1} &P_{j-1\,2} &...&P_{j-1\,n}\\
Q_{1}& Q_{2}&...& Q_{n}\\
P_{j+1\,1} &P_{j+1\,2} &...&P_{j+1\,n}\\
...&...&...&...\\
P_{n1} &P_{n2} &...&P_{nn}
\end{array}
\right|.
$$
Therefore,  we have
$$
\sum\limits_{j=1}^{n}  \Delta_{j}/ \Delta =1,
$$
since $\Delta\neq0.$
Hence, we can treat the above matrix constructions  in the way that gives us the following representation
$$
\sum\limits_{j=1}^{n}  \alpha_{j} P_{j}  =Q,\,\sum\limits_{j=1}^{n} \alpha_{j}   =1,\,\alpha_{j}=\Delta_{j}/ \Delta .
$$
Now, let us prove that $\Upsilon$ belongs to a hyperplane in $\mathbb{E}^{n},$ we have
$$  \left|
\begin{array}{cccc}
P_{11}-Q_{1}&P_{12}-Q_{2}&...&P_{1n}-Q_{n}\\
P_{21}-P_{11}&P_{22}-P_{12}&...&P_{2n}-P_{1n}\\
...&...&...&...\\
P_{n1}-P_{n-1\,1}&P_{n2}-P_{n-1\,2}&...&P_{nn}-P_{n-1\,n}
\end{array}
\right|= \left|
\begin{array}{cccc}
P_{11}&P_{12}&...&P_{1n}\\
P_{21}&P_{22}&...&P_{2n}\\
...&...&...&...\\
P_{n1}&P_{n2}&...&P_{nn}
\end{array}
\right|-
$$
$$
-\left|
\begin{array}{cccc}
 Q_{1}& Q_{2}&...& Q_{n}\\
P_{21}-P_{11}&P_{22}-P_{12}&...&P_{2n}-P_{1n}\\
...&...&...&...\\
P_{n1}-P_{n-1\,1}&P_{n2}-P_{n-1\,2}&...&P_{nn}-P_{n-1\,n}
\end{array}
\right| =
$$
$$
=
 \left|
\begin{array}{cccc}
P_{11}&P_{12}&...&P_{1n}\\
P_{21}&P_{22}&...&P_{2n}\\
...&...&...&...\\
P_{n1}&P_{n2}&...&P_{nn}
\end{array}
\right|-
\left|
\begin{array}{cccc}
\sum\limits_{j=1}^{n}  \alpha_{j} P_{j1}&\sum\limits_{j=1}^{n}  \alpha_{j} P_{j2}&...&\sum\limits_{j=1}^{n}  \alpha_{j} P_{jn}\\
P_{21}-P_{11}&P_{22}-P_{12}&...&P_{2n}-P_{1n}\\
...&...&...&...\\
P_{n1}-P_{n-1\,1}&P_{n2}-P_{n-1\,2}&...&P_{nn}-P_{n-1\,n}
\end{array}
\right| =
$$
$$
 =\left|
\begin{array}{cccc}
P_{11}&P_{12}&...&P_{1n}\\
P_{21}&P_{22}&...&P_{2n}\\
...&...&...&...\\
P_{n1}&P_{n2}&...&P_{nn}
\end{array}
\right|-\sum\limits_{j=1}^{n}  \alpha_{j}
 \left|
\begin{array}{cccc}
     P_{j1}&  P_{j2}&...&  P_{jn}\\
P_{21}-P_{11}&P_{22}-P_{12}&...&P_{2n}-P_{1n}\\
...&...&...&...\\
P_{n1}-P_{n-1\,1}&P_{n2}-P_{n-1\,2}&...&P_{nn}-P_{n-1\,n}
\end{array}
\right| =
$$
$$
 =\left|
\begin{array}{cccc}
P_{11}&P_{12}&...&P_{1n}\\
P_{21}&P_{22}&...&P_{2n}\\
...&...&...&...\\
P_{n1}&P_{n2}&...&P_{nn}
\end{array}
\right|-
 \left|
\begin{array}{cccc}
P_{11}&P_{12}&...&P_{1n}\\
P_{21}&P_{22}&...&P_{2n}\\
...&...&...&...\\
P_{n1}&P_{n2}&...&P_{nn}
\end{array}
\right|\sum\limits_{j=1}^{n}  \alpha_{j} =
$$
$$
 =\left|
\begin{array}{cccc}
P_{11}&P_{12}&...&P_{1n}\\
P_{21}&P_{22}&...&P_{2n}\\
...&...&...&...\\
P_{n1}&P_{n2}&...&P_{nn}
\end{array}
\right|-
 \left|
\begin{array}{cccc}
P_{11}&P_{12}&...&P_{1n}\\
P_{21}&P_{22}&...&P_{2n}\\
...&...&...&...\\
P_{n1}&P_{n2}&...&P_{nn}
\end{array}
\right|  =0.
$$
Hence $\Upsilon$ belongs to a hyperplane generated by the points $P_{i},\,i=1,2,...,n.$ Therefore $ \mathrm{mess} \Upsilon=0,$  and we obtain
  $ \psi =0$ a.e.
The proof is complete.
\end{proof}
Consider     a pre Hilbert  space $ \mathbf{L} ^{n}_{2}(\Omega):=\{f:\,f\in \mathbb{L}^{n}_{2}(\Omega)\}$ endowed with the inner product
$$
(f,g)_{\mathbf{L} ^{n}_{2} }:=\sum\limits_{i=1}^{n}\int\limits_{\Omega} (f ,\mathbf{e_{i}})_{\mathbb{E}^{n}}\overline{(g,\mathbf{e_{i}})}_{\mathbb{E}^{n}} dQ,\,f,g\in \mathbb{L} ^{n}_{2} (\Omega),
$$
where $\mathbf{e_{i}}$ corresponds to $P_{i}\in \partial\Omega,\,i=1,2,...,n,$ condition \eqref{6} holds.
The following theorem establishes a norm equivalence.
\begin{teo}\label{T1}
The   norms $\|\cdot\|_{ \mathbb{L}^{n}_{2} }$ and $\|\cdot\|_{\mathbf{L} ^{n}_{2} } $ are equivalent.
\end{teo}
\begin{proof}
 Consider the space
$
\mathbb{L}^{n}_{2}(\Omega)
$
and a    functional
$
 \varphi(f):= \|f\|_{\mathbf{L} ^{n}_{2} } ,\,f\in \mathbb{L}^{n}_{2}(\Omega).
$
Let us prove that
$
  \varphi(f)\geq C,\,f\in \mathrm{U},
$
where $\mathrm{U}:=\{f\in \mathbb{L}^{n}_{2}(\Omega),\,\|f\|_{\mathbb{L}^{n}_{2}}=1\}.$
Assume the contrary,
then there exists such a sequence $\{\psi_{k}\}_{1}^{\infty}\subset \mathrm{U},$  so that    $\varphi(\psi_{k})\rightarrow0,\,k\rightarrow\infty.$  Since the sequence  $\{\psi_{k}\}_{1}^{\infty}$ is bounded, then we can extract a weekly  convergent subsequence $\{\psi_{k_{j}}\}_{1}^{\infty}$ and  claim that the week limit   $\psi$  of the sequence $\{\psi_{k_{j}}\}_{1}^{\infty}$ belongs to $\mathrm{U}.$  Consider a functional
$$
\mathcal{L}_{g}(f):=\sum\limits_{i=1}^{n}\int\limits_{\Omega} (f ,\mathbf{e_{i}})_{\mathbb{E}^{n}}\overline{(g,\mathbf{e_{i}})}_{\mathbb{E}^{n}} dQ,\;f,g\in \mathbb{L}^{n}_{2}(\Omega).
$$
 Due to the following  obvious  chain of the inequalities
 \begin{equation}\label{7}
|\mathcal{L}_{g}(f)|     \leq \sum\limits_{i=1}^{n}
  \left\{\int\limits_{\Omega}| (f  ,\mathbf{e_{i}})_{\mathbb{E}^{n}}|^{2}   dQ\right\}^{\frac{1}{2}}\left\{\int\limits_{\Omega}| (g  ,\mathbf{e_{i}})_{\mathbb{E}^{n}}|^{2}   dQ\right\}^{\frac{1}{2}}
  \leq
  $$
  $$
  \leq n \|f\|_{\mathbb{L}^{n}_{2}}\|g\|_{\mathbb{L}^{n}_{2}},  \; f,g\in \mathbb{L}^{n}_{2}(\Omega),
  \end{equation}
  we see  that $\mathcal{L}_{g}$ is a linear bounded functional on $\mathbb{L}^{n}_{2}(\Omega).$
Therefore, by virtue of the weak convergence of the sequence $\{\psi_{k_{j}}\},$  we have $\mathcal{L}_{g}(\psi_{k_{j}})\rightarrow \mathcal{L}_{g}(\psi),\,k_{j}\rightarrow \infty.$ On the other hand, recall that since it was  supposed that $\varphi(\psi_{k})\rightarrow0,\,k\rightarrow\infty,$ then we have $\varphi(\psi_{k_{j}})\rightarrow0,\,k\rightarrow\infty.$ Hence applying \eqref{7}, we conclude that $\mathcal{L}_{g}(\psi_{k_{j}})\rightarrow 0,\,k_{j}\rightarrow \infty.$ Combining the given above results we obtain
\begin{equation}\label{8}
\mathcal{L}_{g}(\psi) =\sum\limits_{i=1}^{n}\int\limits_{\Omega} (\psi ,\mathbf{e_{i}})_{\mathbb{E}^{n}}\overline{(g,\mathbf{e_{i}})}_{\mathbb{E}^{n}} dQ=0,\,\forall g \in \mathbb{L}^{n}_{2}(\Omega).
\end{equation}
Taking into account  \eqref{8} and using the ordinary properties of Hilbert space, we obtain
\begin{equation*}
  \sum\limits_{i=1}^{n}\int\limits_{\Omega}| (\psi,\mathbf{e_{i}})_{\mathbb{E}^{n}}|^{2}   dQ  =0.
\end{equation*}
Hence in accordance with Lemma \ref{L3}, we get  $ \psi =0$ a.e.
 Notice  that by virtue of this fact we come to the contradiction with the fact  $\|\psi\|_{\mathbb{L}^{n}_{2}}=1.$
  Hence the following estimate is true
$
  \varphi(f)\geq C,\,f\in \mathrm{U}.
$
Having applied the Cauchy Schwartz inequality to the Euclidian inner product, we can also easily obtain
$
 \varphi(f) \leq \sqrt{n}\|f\|_{\mathbb{L}^{n}_{2}},\,f\in \mathbb{L}^{n}_{2}(\Omega).
$
Combining the above inequalities,   we can rewrite these two estimates as follows
$
C_{0}\leq\varphi(f)\leq C_{1},\,f\in \mathrm{U}.
$
To make the issue clear,  we can  rewrite the previous inequality in the form
\begin{equation}\label{9}
C_{0}\|f\|_{\mathbb{L}^{n}_{2} }\leq\varphi(f)\leq C_{1}\|f\|_{\mathbb{L}^{n}_{2} },\,f\in  \mathbb{L}^{n}_{2}(\Omega),\;C_{0},C_{1}>0.
\end{equation}
The proof is complete.
\end{proof}
Consider a pre Hilbert  space
$$
\mathfrak{\widetilde{H}}^{n}_{ A }:= \big \{f,g\in C_{0}^{\infty}(\Omega),\,(f,g)_{\mathfrak{\widetilde{H}}^{n}_{ A }}=\sum\limits_{i=1}^{n}(A_{i} f,A_{i} g)_{L_{2} } \big\},
$$
where $-A_{i}$ is an infinitesimal generator corresponding to the point $P_{i}.$ Here, we should point out that the form $(\cdot,\cdot)_{\mathfrak{\widetilde{H}}^{n}_{ A }} $ generates an inner product due to the fact $\mathrm{N}(A_{i})=0,\,i=1,2,...,n$ proved in Lemma \ref{L2}.
  Let us denote a corresponding Hilbert space by $\mathfrak{H}^{n}_{A}.$
\begin{corol}\label{C1}
  The norms $\|\cdot\|_{\mathfrak{H}^{n}_{A}}$ and $\|\cdot\|_{H_{0}^{1}} $ are equivalent, we have a bounded compact  embedding
$$
\mathfrak{H}^{n}_{A}\subset\subset L_{2}(\Omega).
$$
\end{corol}
\begin{proof} Let us prove that
\begin{equation*}
 Af =                -(\nabla f ,\mathbf{e})_{\mathbb{E}^{n}},\,f\in  C^{\infty}_{0}( \Omega ).
\end{equation*}
Using the Lagrange  mean value  theorem, we have
$$
\int\limits_{\Omega}\left|  \left(\frac{T_{t}-I}{t}\right)f(Q)-(\nabla f ,\mathbf{e})_{\mathbb{E}^{n}}(Q )\right|^{2} dQ=\int\limits_{\Omega}\left|  (\nabla f ,\mathbf{e})_{\mathbb{E}^{n}}(Q_{\xi} )-(\nabla f ,\mathbf{e})_{\mathbb{E}^{n}}(Q )\right|^{2} dQ,
$$
where $Q_{\xi}=Q+ \mathbf{e}  \xi,\,0<\xi<t.$ Since the function $(\nabla f ,\mathbf{e})_{\mathbb{E}^{n}}$ is continuous on $\bar{\Omega},$ then it is uniformly continuous on $\bar{\Omega}.$ Thus,  for arbitrary $\varepsilon>0,$ a positive number $\delta>0$ can be chosen so that
$$
 \int\limits_{\Omega}\left|  (\nabla f ,\mathbf{e})_{\mathbb{E}^{n}}(Q_{\xi} )-(\nabla f ,\mathbf{e})_{\mathbb{E}^{n}}(Q ),\right|^{2} dQ<\varepsilon,\,t<\delta,
$$
from what follows the desired result. Taking it into account, we  obtain
 $$
 \|Af\|_{L_{2}}=\left\{\int\limits_{\Omega}|(\nabla f ,\mathbf{e})_{\mathbb{E}^{n}}|^{2}dQ\right\}^{1/2}\leq \left\{\int\limits_{\Omega} \|\mathbf{e}\|^{2}_{\mathbb{E}^{n}}\sum\limits_{i=1}^{n}|D_{i}f|^{2} dQ\right\}^{1/2} =\|f\|_{H_{0}^{1}},\,f\in C^{\infty}_{0}(\Omega).
 $$
Using this estimate, we easily obtain  $\|f\|_{\mathfrak{H}^{n}_{A}}\leq C \|f\|_{H_{0}^{1}},\,f\in  C_{0}^{\infty}(\Omega).$ On the other hand,
as a particular case of formula \eqref{9}, we obtain
$
C_{0}\|f\|_{H_{0}^{1}}\leq\|f\|_{\mathfrak{H}^{n}_{A}} ,\,f\in  C_{0}^{\infty}(\Omega).
$
Thus, we can combine  the previous inequalities and rewrite them as follows
$
C_{0}\|f\|_{H_{0}^{1}}\leq\|f\|_{\mathfrak{H}^{n}_{A}}\leq C \|f\|_{H_{0}^{1}},\,f\in  C_{0}^{\infty}(\Omega).
$
 Passing to the limit at the left-hand and right-hand side of the last inequality, we get
\begin{equation*}
C_{0}\|f\|_{H_{0}^{1}}\leq\| f\|_{\mathfrak{H}^{n}_{A}}\leq C \|f\|_{H_{0}^{1}},\,f\in  H_{0}^{1}(\Omega).
\end{equation*}
Combining the  fact
 $
 H_{0}^{1}(\Omega)\subset\subset L_{2}(\Omega),
 $
 (Rellich-Kondrashov theorem)
 with the  lower estimate in the previous inequality,
 we complete the proof.
\end{proof}
\vspace{0.5 cm}

\noindent{\bf Uniformly elliptic operator in the divergent form}\\

Consider a uniformly ecliptic operator
$$
\,-\mathcal{T}:=-D_{j} ( a^{ij} D_{i}\cdot),\, a^{ij}(Q) \in C^{2}(\bar{\Omega}),\,   a^{ij}\xi _{i}  \xi _{j}  \geq   \gamma_{a}  |\xi|^{2} ,\,  \gamma_{a}  >0,\,i,j=1,2,...,n,\;
$$
$$
\mathrm{D}( \mathcal{T} )  =H^{2}(\Omega)\cap H^{1}_{0}( \Omega ).
$$
The following theorem gives us a key to apply results of the paper \cite{kukushkin Gen} in accordance with  which a number of spectral theorems  can be applied to the operator $-\mathcal{T}.$ Moreover the conditions established bellow are formulated in terms of the operator $A,$ what reveals a mathematical nature of the operator $-\mathcal{T}.$
\begin{teo}\label{T2}
We claim that
\begin{equation}\label{10}
 -\mathcal{T}=\frac{1}{n}\sum\limits^{n}_{i=0} A_{i}^{\ast}G_{i}A_{i},
\end{equation}
the following relations hold
$$
-\mathrm{Re}(\mathcal{T}f,f)_{L_{2}}\geq C  \|f\|_{\mathfrak{H}^{n}_{A}};\, |(\mathcal{T}f,g)_{L_{2}}|\leq C \|f\|_{\mathfrak{H}^{n}_{A}}\|g\|_{\mathfrak{H}^{n}_{A}},\;f,g\in C_{0}^{\infty}(\Omega),
$$
where $G_{i}$ are some operators corresponding to the operators $A_{i}.$
\end{teo}
\begin{proof}
It is easy to prove  that
\begin{equation}\label{11}
\|A_{i}f\|_{L_{2}}\leq C\|f\|_{H_{0}^{1}},\,f\in H_{0}^{1}(\Omega),
\end{equation}
 for this purpose we should use a representation $A_{i}f(Q) =-(\nabla f ,\mathbf{e_{i}})_{\mathbb{E}^{n}}, f\in C^{\infty}_{0}(\Omega).$ Applying  the Cauchy Schwarz inequality, we get
 $$
 \|A_{i}f\|_{L_{2}}\leq\left\{\int\limits_{\Omega}|(\nabla f ,\mathbf{e_{i}})_{\mathbb{E}^{n}}|^{2}dQ\right\}^{1/2}\leq \left\{\int\limits_{\Omega}\|\nabla f\|^{2}_{\mathbb{E}^{n}}
 \|\mathbf{e_{i}}\|^{2}_{\mathbb{E}^{n}} dQ\right\}^{1/2}=\|f\|_{ H_{0}^{1}},\,f\in C^{\infty}_{0}(\Omega).
 $$
 Passing to the limit at the left-hand and right-hand side, we obtain \eqref{11}.
Thus, we get   $H_{0}^{1}(\Omega) \subset \mathrm{D}(A_{i}).$
Let us find a representation for the  operator $G_{i}.$
Consider the operators
 $$
 B_{i}f(Q)=\!\int_{0}^{r}\!\!f(P_{i}+\mathbf{e }[r-t])dt,\,f\in L_{2}(\Omega),\,i=1,2,...n.
 $$
It is obvious that
\begin{equation}\label{12}
\int\limits_{\Omega }  A_{i}\left(B_{i}  \mathcal{T}  f \cdot g\right) dQ =\int\limits_{\Omega } A_{i}B_{i}\mathcal{T}f \cdot g\,  dQ +\int\limits_{\Omega }  B_{i} \mathcal{T} f \cdot A_{i}g  \,dQ ,\, f\in C^{2}(\bar{\Omega}),g\in C^{\infty}_{0}( \Omega ).
\end{equation}
Using the   divergence  theorem, we get
\begin{equation}\label{13}
\int\limits_{\Omega }  A_{i}\left(B_{i}\mathcal{T}f \cdot g\right) \,dQ=\int\limits_{S}(\mathbf{e_{i}},\mathbf{n})_{\mathbb{E}^{n}}(B_{i}\mathcal{T}f\cdot  g)(\sigma)d\sigma,
\end{equation}
where $S$ is the surface of $\Omega.$
Taking into account   that $ g(S)=0$ and combining   \eqref{12},\eqref{13}, we get
\begin{equation}\label{14}
    -\int\limits_{\Omega }  A_{i}B_{i}\mathcal{T} f\cdot \bar{g} \, dQ=  \int\limits_{\Omega } B_{i}\mathcal{T} f\cdot   \overline{A_{i}  g} \, dQ,\, f\in C^{2}(\bar{\Omega}),g\in C^{\infty}_{0}( \Omega ).
\end{equation}
Suppose that $f\in H^{2}(\Omega),$ then    there exists a sequence $\{f_{n}\}_{1}^{\infty}\subset C^{2}(\bar{\Omega})$ such that
$ f_{n}\stackrel{ H^{2}}{\longrightarrow}  f$ (see \cite[p.346]{firstab_lit:Smirnov5}).  Using this fact, it is not hard to prove that
$\mathcal{T}f_{n}\stackrel{L_{2}}{\longrightarrow} \mathcal{T}f.$  Therefore  $A_{i}B_{i}\mathcal{T}f_{n}\stackrel{L_{2}}{\longrightarrow} \mathcal{T}f,$  since  $A_{i}B_{i}\mathcal{T}f_{n}=\mathcal{T}f_{n}.$ It is also clear that   $B_{i}\mathcal{T}f_{n}\stackrel{L_{2}}{\longrightarrow} B_{i}\mathcal{T}f,$ since $B_{i}$ is continuous (see proof of Lemma \ref{L2}).
Using these facts, we can extend relation \eqref{14} to the following
\begin{equation}\label{15}
  -\int\limits_{\Omega }  \mathcal{T} f \cdot  \bar{g} \, dQ= \int\limits_{\Omega } B_{i}\mathcal{T}f\,  \overline{A_{i}g} \, dQ,\; f\in \mathrm{D}(\mathcal{T}),\,g\in C_{0}^{\infty}(\Omega).
\end{equation}
Note, that it was previously proved that $A_{i}^{-1} \subset B_{i}$ (see the proof of Lemma \ref{L2}),   $H_{0}^{1}(\Omega) \subset \mathrm{D}(A_{i}).$ Hence $G_{i} A_{i}f=B_{i}\mathcal{T} f,\, f\in  \mathrm{D}(\mathcal{T}),$  where $G_{i}:=B_{i}\mathcal{T}B_{i}.$  Using  this fact    we can rewrite relation \eqref{15} in a   form
\begin{equation}\label{16}
  -\int\limits_{\Omega }  \mathcal{T} f \cdot  \bar{g} \, dQ= \int\limits_{\Omega } G_{i} A_{i}f\,  \overline{A_{i}g} \, dQ,\; f\in \mathrm{D}(\mathcal{T}),\,g\in C_{0}^{\infty}(\Omega).
\end{equation}
Note that   in accordance with Lemma \ref{L2}, we have
$$
\forall g\in \mathrm{D}(A_{i}),\,\exists \{g_{n}\}_{1}^{\infty}\subset C^{\infty}_{0}( \Omega ),\,    g_{n}\xrightarrow[      A_{i}       ]{}g.
$$
Therefore, we can extend   relation \eqref{16} to the following
\begin{equation}\label{17}
     -\int\limits_{\Omega }  \mathcal{T} f \cdot  \bar{g} \, dQ= \int\limits_{\Omega } G_{i} A_{i}f \, \overline{A_{i}g} \, dQ,\; f\in \mathrm{D}(\mathcal{T}),\,g\in \mathrm{D}(A_{i}).
\end{equation}
Relation \eqref{17} indicates that $G_{i} A_{i}f\in \mathrm{D}( A_{i} ^{\ast})$   and it is clear that $  -\mathcal{T}\subset A_{i} ^{\ast}G_{i}A_{i}.$ On the other hand in accordance with  Chapter VI, Theorem 1.2    \cite{firstab_lit: Berezansk1968}, we have that $-\mathcal{T}$ is a closed operator.
Using the divergence theorem we get
$$
- \int\limits_{\Omega}D_{j} ( a^{ij} D_{i}f) \bar{g} dQ=
  \int\limits_{\Omega}    a^{ij}  D_{i}f\,     \overline{D_{j}g}    dQ,\,f\in C^{2}(\Omega),\,g\in C^{\infty}_{0}(\Omega).
$$
Passing to the limit at the left-hand and right-hand side of  the last inequality, we  can   extend it to the following
$$
- \int\limits_{\Omega}D_{j} ( a^{ij} D_{i}f)\, \bar{g} dQ=
  \int\limits_{\Omega}    a^{ij}  D_{i}f\,     \overline{D_{j}g}    dQ,\,f\in H^{2}(\Omega),\,g\in H^{1}_{0}(\Omega).
$$
Therefore, using the  uniformly elliptic property of the operator $-\mathcal{T},$ we get
 \begin{equation}\label{18}
-\mathrm{Re} \left(\mathcal{T}f,  f \right)_{L_{2}} \geq \gamma_{a}\int\limits_{\Omega}     \sum\limits_{i=1}^{n}   |  D_{i}f |^{2}    \,    dQ=
\gamma_{a}\|f\|^{2}_{H_{0}^{1}},\,f\in \mathrm{D}(\mathcal{T}).
\end{equation}
Using    the  Poincar\'{e}-Friedrichs inequality, we get
$
-\mathrm{Re} \left(\mathcal{T}f,  f \right)_{L_{2}}\geq C\|f\|^{2}_{L_{2}},\,f\in \mathrm{D}(\mathcal{T}),
$
Applying the Cauchy-Schwarz inequality to the left-hand side, we can easily deduce that the conditions of Lemma \ref{L1} are satisfied.
 Thus,   the operator  $-\mathcal{T}$ is m-accretive. In particular, it means that there does not exist an accretive extension of the operator $-\mathcal{T}.$ Let us prove that $A_{i} ^{\ast}G_{i}A_{i}$ is accretive, for this purpose combining  \eqref{16},\eqref{18}, we get
$
  \left(G_{i} A_{i}f , A_{i}f \right)_{L_{2}}  \geq 0,  f \in C_{0}^{\infty}  (\Omega).
$
Due to the relation $\tilde{A}_{0}=A,$ proved in Lemma \ref{L2}, the previous  inequality can be easily extended  to
$
 \left(G_{i} A_{i}f , A_{i}f \right)_{L_{2}}  \geq 0,  f \in \mathrm{D}(G_{i} A_{i}).
$
In its own    turn, it implies that
$
\left( A^{\ast}_{i}G_{i} A_{i}f ,f \right)_{L_{2}}  \geq 0,  f \in \mathrm{D}(A^{\ast}_{i}G_{i} A_{i}),
$
thus  we have obtained the desired result. Therefore, taking into account the facts given above,  we deduce that
  $-\mathcal{T}= A_{i}  ^{\ast}G_{i}A_{i},\,i=1,2,...\,n\,$ and     obtain   \eqref{10}.
Applying the Cauchy-Schwarz inequality to the inner sums, then using   Corollary \ref{C1}, we obtain
 \begin{equation*}
\left|\int\limits_{\Omega }  \mathcal{T} f \cdot  \bar{g} \, dQ\right|=\left|\int\limits_{\Omega}    a^{ij}  D_{i}f\,     \overline{D_{j}g}    dQ\right|\leq a_{1}
 \int\limits_{\Omega}      \|\nabla f\| _{\mathbb{E}^{n}}\,     \|\nabla g\| _{\mathbb{E}^{n}}    dQ \leq
 $$
 $$
 \leq a_{1}\|f\|_{H_{0}^{1}}\|g\|_{H_{0}^{1}}\leq C\|f\|_{\mathfrak{H}^{n}_{A}}\|g\| _{\mathfrak{H}^{n}_{A}} ,\; f,g\in C_{0}^{\infty}(\Omega),
 $$
 where
 $$
 a_{1}=\sup\limits_{Q\in \bar{\Omega}} \sqrt{\sum\limits_{i,j=1}^{n} |a^{ij}(Q)|^{2}}.
\end{equation*}
 On the other hand, applying  \eqref{11},\eqref{18} we get
\begin{equation*}
-\mathrm{Re}(\mathcal{T} f,f) \geq C\|f\|^{2}_{\mathfrak{H}^{n}_{A}},\,f\in C^{\infty}_{0}(\Omega).
\end{equation*}
  The proof is complete.
\end{proof}
Thus, by virtue of     Corollary \ref{C1} and Theorem \ref{T2}, we are able to   claim that  theorems $(\mathbf{A})-(\mathbf{C})$  \cite{kukushkin Gen} can be applied to the operator $-\mathcal{T}.$

\section{Conclusions}

In this paper we have  established a norm equivalence in the Lebesgue  space,  it gives us an opportunity to reveal more fully a true mathematical nature of a differential operator.
As a consequence  of the mentioned equivalence  we have  a   compact embedding of a space generated by the infinitesimal generator of the shift semigroup in a direction into the Lebsgue space.  The considered particular  case corresponds to a uniformly elliptic operator which is  not  selfadjoint  under minimal assumptions regarding its coefficients. Thus,  the opportunity to apply spectral theorems  in the natural way   becomes  relevant, since there are not many  results devoted to the spectral properties of non-selfadjoint operators. Along with all these,  by virtue of popularity and the well-known applicability of the  Lebesgue spaces theory, the result related to the norm  equivalence   deserves    to be considered    itself.

\end{document}